\documentclass[english, a4paper, reqno]{smfart}

\usepackage[utf8]{inputenc}
\usepackage[all]{xy}
\usepackage{amsbsy, amssymb}
\usepackage{a4}
\usepackage{fullpage}
\usepackage{microtype}
\usepackage{smfthm}
\usepackage[english]{babel}

\providecommand{\R}{\mathbb{R}}
\providecommand{\wb}{\mathbb{W}}
\providecommand{\C}{\mathbb{C}}

\providecommand{\T}{\mathbb{T}}
\providecommand{\PP}{\mathbb{P}}

\providecommand{\G}{\mathbb{G}}
\providecommand{\Sp}{\mathbb{S}}

\providecommand{\J}{\mathbb{J}}

\providecommand{\vc}{\mathcal{V}}
\providecommand{\hc}{\mathcal{H}}
\providecommand{\lra}{\longrightarrow}
\providecommand{\lms}{\longmapsto}
\providecommand{\bw}{\bigwedge}
\providecommand{\w}{\wedge}
\providecommand{\wh}{\widehat} 
\providecommand{\hh}{\mathcal{H}} 
\providecommand{\ov}{\overline}
\providecommand{\zb}{\overline{z}}

\providecommand{\vol}{\mathop{Vol}\nolimits}
\providecommand{\hess}{\mathop{i d'd''}\nolimits}
 \providecommand{\SO}{\mathop{SO}}
\providecommand{\End}{\mathop{End}}

\newtheorem*{theoA}{Theorem A}\newtheorem*{theoB}{Theorem B}\newtheorem*{theoC}{Theorem C}
\newtheorem*{theoD}{Theorem D}

\begin{document}
\frontmatter
\title{Hessian of the natural Hermitian form on twistor spaces}
\author{Guillaume Deschamps}
\address{Université de Brest}
\email{Guillaume.Deschamps@univ-brest.fr}
\author{No\"el Le Du}
\address{ Université Rennes 1}
\email{noellouis.ledu@orange.fr}
\author{Christophe Mourougane}
\address{ Université Rennes 1}
\email{christophe.mourougane@univ-rennes1.fr}
\date{\today}

\begin{abstract}
We compute the hessian $\hess\wb$ of the natural Hermitian form~$\wb$ successively 
on the Calabi family $\T(M,g,(I,J,K))$ of a hyperk\"ahler manifold $(M,g,(I,J,K))$, 
on the twistor space $\T(M,g)$ of a $4$-dimensional anti-self-dual Riemannian manifold $(M,g)$
and on the twistor space $\T(M,g,D)$ of a quaternionic Kähler manifold $(M,g,D)$.
We show  a strong convexity property of the cycle space of twistor lines on the Calabi family $\T(M,g,(I,J,K))$ of a hyperk\"ahler manifold.
We also prove convexity properties of the $1$-cycle space of the twistor space $\T(M,g)$ of a $4$-dimensional anti-self-dual Einstein manifold $(M,g)$ of non-positive scalar curvature
and of the $1$-cycle space of the twistor space $\T(M,g,D)$ of a quaternionic Kähler manifold $(M,g,D)$ of non-positive scalar curvature.
We check that no non-K\"ahler strong KT manifold occurs as such a twistor space.
\end{abstract}

\maketitle
\setcounter{tocdepth}{4}
\tableofcontents

\mainmatter

\section{Introduction}
The twistor construction is known to provide examples of manifolds
 endowed with a natural metric $\G$ and a natural almost complex structure $\J$
that is sometimes integrable and often non-K\"ahler~(see section~\ref{def} for precise definitions).
Our aim is to compute the exterior derivative and the Hessian of the natural Hermitian form $\wb=\G(\J\cdot,\cdot)$
for different twistor constructions. 

We derive, under compactness and non-positive scalar curvature assumption for the base space,
a convexity property for the $1$-cycle space $C_1(\T)$ of the twistor space $\T$,
that could be a substitute to the well known compactness of the cycle space of compact K\"ahler manifolds.

The classical twistor construction is for anti-self-dual Riemannian $4$-manifolds.
We can here in full generality compute the Hessian of the natural Hermitian form (see theorem~\ref{theo1}). 
Under extra assumptions on the base Riemannian manifold, we can study the convexity properties of the cycle space $C_1(\T)$.
\begin{theoA}(corollary~\ref{cor:einstein-convex})
 The hessian $\hess\wb$ of the Hermitian form $\wb$ on the twistor
space $\T=\T(M,g)$ of a $4$-dimensional Einstein manifold $(M,g)$ with non-positive constant scalar curvature $s$
is non-negative. If furthermore $M$ is compact, the volume function on the $1$-cycle space $C_1(\T)$ is a continuous pluri-sub-harmonic exhaustion function.
\end{theoA}

A similar construction can be made starting with a higher dimensional Riemannian manifold with quaternionic holonomy.
A \textit{quaternionic Kähler manifold} is an oriented complete $4n$-dimensional Riemannian
manifold $(M,g)$ whose holonomy group is contained in the product $Sp(1)Sp(n)$ of quaternionic unitary groups.
Such a manifold admits a rank $3$ sub-bundle $D\subset \End(TM)$ invariant by the Levi-Civita connection of $(M,g)$, 
locally spanned by a quaternionic triple $(I,J,K=IJ=-JI)$ 
of almost complex structures $g$-orthogonal and compatible with the orientation.
One can define \textit{the twistor space $\pi~:~\T=\T(M,g,D)\to M$} as the bundle of spheres of radius $\sqrt{2}$ of $D$.

In the case of positive scalar curvature, the manifold $M$ is compact 
and Salamon (\cite[theorem 6.1]{Sal82}) showed that its twistor space admits a Kähler-Einstein metric of positive scalar curvature,
that co\"incides with the metric $\G$, up to changing the choice for the radius of vertical spheres.
In particular, $\T$ is a compact complex manifold with positive first Chern class, that is a Fano manifold.
The projection onto the vertical direction gives a contact structure~(\cite[theorem 4.3]{Sal82}).
By the Kähler property of $\T$, every component of its cycle space is compact.

In the case of negative scalar curvature, the twistor space is a complex contact uniruled manifold.
The only known compact examples are locally symmetric. 
We show in this case that the components of the $1$-cycle space are pseudo-convex. 
More precisely, we find the
\begin{theoB}(corollary~\ref{pseudo-convex})\label{theo-A}
The hessian $\hess\wb$ of the Hermitian form $\wb$ on the twistor
space $\T=\T(M,g,D)$ of a quaternionic Kähler $4n$-manifold $(M,g,D)$ with non-positive constant scalar curvature $s$
is semi-positive. If furthermore $M$ is compact, the volume function on the $1$-cycle space is a continuous pluri-sub-harmonic exhaustion function.
\end{theoB}

In the case of zero scalar curvature, the manifold $M$ is in fact locally hyperkähler. 
A \textit{hyperkähler manifold} is an oriented $4n$-dimensional Riemannian manifold $(M,g)$ 
whose holonomy group is contained in the quaternionic unitary group $Sp(n)$.
In other words, a hyperkähler manifold is an oriented $4n$-dimensional Riemannian
manifold $(M,g)$ endowed with a quaternionic triple of global $g$-Kähler complex structures $I$, $J$ and $K$ compatible with the orientation.
The corresponding pencil of complex structures $f~:~\T=\T(M,g,D)\to\PP^1$ is called the \textit{Calabi family} of $(M,g,D=(I,J,K))$.
In this case, we can relate the non-K\"ahler feature of $\T(M,g,D)$ with the Kodaira-Spencer class of the pencil $f$.
\begin{theoC}(theorem~\ref{domega-k})
Let $(\theta_1,\ldots,\theta_{4n})$ be a local orthonormal frame of $TM$. For a vertical vector $U\in\vc_{(m,u)}$,
\begin{eqnarray*}
d''\wb_{(m,u)}(U^h,\hh\theta_i^a,\hh\theta_j^a)&=&-2\Omega_u(\kappa_U(\theta_i^a),\kappa_U(\theta_j^a))
\end{eqnarray*}
where $\Omega_u$ is the holomorphic symplectic $(2,0)$-form on $X_u:=f^{-1}(u)$ 
and $\kappa_U$ is a closed $(0,1)$-form on $X_u$
with values in $TX_u$ that represents the Kodaira-Spencer class of the family $f$ at $u\in\PP^1$
in the direction $U$.
\end{theoC}
With second order derivatives, we get a precise control on the volume function for $1$-cycles of $\T$
deformations of the twistor lines
\begin{theoD}(theorem~\ref{plurif})
Let  $(M,g,D=(I,J,K))$ be a compact hyperkähler manifold.
Let  $C^0_1(\T)$ be the component of the Barlet cycle space of $\T(M,g,D)$ containing the twistor lines.
The map $\vol : C^0_1(\T)\to\R$ is a continuous pluri-sub-harmonic exhaustion function. More precisely,
\begin{eqnarray*}
\hess_{C^0_1(\T)}\vol (C_s)(\vec{n},J\ov{\vec{n}})\geq\int_{C'_s} \Vert n'\Vert ^2 d vol\geq 0
\end{eqnarray*}
where $C'_s$ is the irreducible component of the cycle $C_s$ that maps onto $\PP^1$ by the pencil map $f$.
In particular, the cycle space $C^0_1(\T)$ is pseudo-convex.
\end{theoD}

For example, starting with a compact \textit{holomorphic symplectic manifold} $(X,\Omega)$ of complex dimension~$2n$
and a Kähler class $\kappa$, by the theorem of Yau~\cite{yau} we get a Ricci flat metric $g$ with $Sp(n)$ holonomy.
The corresponding twistor space $f~:~\T(X,\Omega,\kappa)\to\PP^1$ is called \textit{the Calabi family} of $(X,\Omega,\kappa)$.
This construction and the component $C_1(\T)$ of the cycle space were used by Campana
to show that in every Calabi family one member contains a non-constant entire curve~\cite{Cam91, Cam92}.
This work has recently been pushed further by Verbitsky~\cite{Verb}, to show that every compact holomorphic symplectic 
manifold contains a non-constant entire curve, (that is, is not Kobayashi hyperbolic) and by~\cite{Verb2}
even further to show that the Kobayashi pseudo metric
 vanishes for all know examples except if their Picard number is maximal. 
We expect that the formula in theorem~\textbf{D.} could help to localise 
rational curves on compact holomorphic symplectic manifolds.

Complex Hermitian manifolds with $\hess$-closed Hermitian form are called \textit{strong K\"ahler with torsion (strong KT)}.
Constructing examples often starts with group theoretical considerations~\cite{fino-salamon, fino-review}. 
It could be expected that twistor constructions could also provide interesting examples.
We check from our computations that the vanishing of $\hess\wb$ amounts to that of $d\wb$.
Hence, no non-K\"ahler strong KT space can be constructed with this natural metric on all the considered twistor spaces.

In the text, we first deal in section~\ref{sec:1} with hyperk\"ahler manifolds,
where the pencil map $f~:~\T=\T(M,g,D)\to\PP^1$ hugely simplifies the computations.
We then turn in section~\ref{sec:2}, to the case of $4$-dimensional anti-self-dual Riemannian manifolds,
that displays all the important features of this aspect of twistor geometry.
The last section~\ref{sec:3} on quaternionic K\"ahler manifolds parallels 
the previous one under the additional Einstein assumption. 
The section~\ref{def} provides the basics on the twistor constructions.

\textit{The authors would like to thank Frédéric Campana, Benoît Claudon, Daniel Huybrechts and Misha Verbitsky for 
useful discussions and remarks.}

\section{Preliminaries on twistor constructions}
\label{def}
\subsection{Constructions on $\R^4$}

An endomorphism $u$ of the oriented Euclidean real vector space $\R^4$ is said to \textit{respect
the orientation} if for all vectors $X,Y\in\R^4$ the $4$-tuple
$(X,uX,Y,uY)$ is either linearly dependent or positively oriented. This will be denoted by $u\gg 0$.
Examples are given by the following three orthogonal anti-involutive (hence anti-symmetric) endomorphisms
$$
I=\left[\begin{array}{cccc}
0&-1&0&0\\
1&0&0&0\\
0&0&0&-1\\
0&0&1&0\end{array}\right], \ \  J=\left[\begin{array}{cccc}
0&0&-1&0\\
0&0&0&1\\
1&0&0&0\\
0&-1&0&0\end{array}\right], \ \  K=\left[\begin{array}{cccc}
0&0&0&-1\\
0&0&-1&0\\
0&1&0&0\\
1&0&0&0\end{array}\right].
$$

The set $$F:=\{u\in SO(4), u^2=-Id, u\gg0\}$$
of complex structures on $\R^4$ that respect the orientation and the Euclidean product
(called compatible complex structures)
identifies with the sphere $\{u=aI+bJ+cK/(a,b,c)\in\Sp^2\}\simeq\Sp^2$.
At a point $u\in F$, the tangent space $T_uF$ is $\{A\in so(4)/ Au+uA=0\}$
The standard metric $g_0$ on the sphere of radius $\sqrt{2}$ reads on $F$
$$
g_0(A, B)=\frac{1}{2}tr(A{}^tB)=-\frac{1}{2}tr(AB),\qquad\forall A,B\in T_uF.
$$
 As the sphere $\Sp^2$, the set $F$ inherits the complex structure of $\C P^1$.
More precisely, the complex structure of $T_uF$ reads $j\cdot A=uA$ as a matrix product.

This identification can be made intrinsic as follows.
The Euclidean product on $\R^4$ gives an Euclidean product
on the exterior product $\bigwedge^2\R^4$.
The Hodge star operator splits  $\bigwedge^2\R^4$ into $\bigwedge^2\R^4=\bigwedge^+\oplus\bigwedge^-$.
An anti-symmetric endomorphism $A$ of $so(4)$
identifies with an element $\phi(A)$ of $\bigwedge^2\R^4$ via
 $$g(\phi(A), X\wedge Y)=g(A X,Y)\qquad\forall X,Y\in \R^4.$$
For example, if $(e_j)$ is an orthonormal basis of $\R^4$, the anti-symmetric endomorphism associated with $e_k\wedge e_l$
sends $e_k$ to $e_l$ and $e_l$ to $-e_k$ and all other base vector $e_j$ to $0$.
In particular, a compatible complex structure $u$ identifies with an element $\phi(u)$
precisely of the  sphere of vectors of $\bigwedge^+$ of norm $\sqrt{2}$.
We will always identify $\phi(u)$ with $u$ and $\bigwedge^2 \R^4$ with~$so(4)$.

\subsection{Constructions on a Riemannian $4$-manifold}
Consider now a $4$-di\-men\-sional oriented Riemannian manifold $(M,g)$.
Its twistor space is the fibre bundle 
$$\xymatrix{
   \T(M,g)=\T \ar[d]^{\pi}\\
 M }$$   of
vectors of $\bigwedge^+TM=:\bw^+$ of norm $\sqrt{2}$. Fibre-wise, it identifies
with the set of compatible complex structures on the tangent space of $M$.

A natural Riemannian metric $\G$ and a natural almost-complex structure $\J$
are defined on the twistor space $\T$ as follows.
The bundle $\vc$ of vertical directions in $T\T$ is the kernel of $d\pi$.
Note that its structure group is $SO(3)\subset PGL(2,\C)$
so that fibres inherit complex structures and Riemannian metrics.
The Levi-Civita connection $\nabla^g$ of $(M,g)$ provides us
with a bundle $\hc$ of horizontal directions in $T\T$
isomorphic via $d\pi$ with $TM$, hence endowed with a complex structure and a Riemannian metric.
The natural Riemannian metric $\G$ and the natural almost-complex structure $\J$ on $\T$ are defined
so that they coïncide on the summand of the decomposition 
$$T\T=\hc\oplus\vc$$
 with the previous structures
and the decomposition is made $\G$-orthogonal and invariant by $\J$.
The map $\pi$ becomes a Riemannian submersion and its fibres are rational curve with Fubini Study metrics.
We will study the natural Hermitian form $\wb=\G(\J\cdot,\cdot)$.

We will use the notation $\hc(X)=\hc_p(X)$ to denote the horizontal lift at $p\in\pi^{-1}(m)$
 of a vector $X$ tangent to $M$ at $m$,
and likewise the notation $\hc(V)$ to denote the (orthogonal) projection of a vector $V$ tangent to $\T$
onto its horizontal part along the vertical direction.

For a real tangent vector $V$ on $\T$, we will denote by $V^h:=\frac{V-iJV}{2}, V^a:=\frac{V+iJV}{2}\in T\T_\C$
 its $(1,0)$ and $(0,1)$ parts~: $\J V^h=i V^h$, $\J V^a=-iV^a$.
Moreover, for a real tangent vector $X$ on $M$, $X^h$   will denote
$$X^h:=\pi_\star(\hh{(X)}^h)=1/2\pi_\star(\hh X-i\J\hh X)=1/2(X-iu(X))$$
and $X^a=\pi_\star(\hh{(X)}^a)$, omitting the dependence on $p=(m,u)\in \pi^{-1}(m)$.
Note that by the construction of $\J$, one has $\hc(X)^h=\hc(X^h)=:\hc X^h$ and $\hc(X)^a=\hc(X^a)=:\hc X^a$.

\subsection{Constructions on a quaternionic Kähler manifold}
Fix an integer $n\geq 1$. A \textit{quaternionic Kähler manifold} is an oriented complete $4n$-dimensional Riemannian
manifold $(M,g)$ whose holonomy group is contained in the product $Sp(1)Sp(n)$ of quaternionic unitary groups.
In other words, with the holonomy principle \cite{Bes87},
such a manifold admits a rank $3$ sub-bundle $D\subset \End(TM)$ invariant by the Levi-Civita connection of $(M,g)$, 
locally spanned by a quaternionic triple $(I,J,K=IJ=-JI)$ 
of almost complex structures $g$-orthogonal and compatible with the orientation.
We will use the notation $\nabla:=\nabla^{(g,D)}$ for the restriction to $D$ of the Levi-Civita connection,
and subsequently $R$ for the curvature of this restriction.
Berger proved that quaternionic Kähler manifolds are Einstein (\cite{berger} see also~\cite[theorem 14.39]{Bes87}).

Let $(M,g,D)$ be a  quaternionic Kähler $4n$-manifold.
One can define its \textit{twistor space $\pi~:~\T=\T(M,g,D)\to M$} as the bundle of spheres of radius $\sqrt{2}$ of $D$.
This is a locally trivial bundle over $M$ with fibre $\Sp^2$ and structure group $SO(3)$.
Using the splitting of the tangent bundle $T\T$ given by the Levi-Civita connection of $(M,g)$,
the twistor space $\T$ can be endowed with a metric $\G$ and an almost complex structure $\J$
that is integrable~(\cite[theorem 4.1]{Sal82},\cite{hklr}).

The previous remarks for horizontal lifts and decomposition in types hold.

\subsection{Constructions on a hyperkähler manifold} 
We now describe a special case of the previous construction. Fix an integer $n\geq 1$.
Recall that a \textit{hyperkähler manifold} is an oriented $4n$-dimensional Riemannian manifold $(M,g)$ 
whose holonomy group is contained in the quaternionic unitary group $Sp(n)$.
In other words, with the holonomy principle, a hyperkähler manifold is an oriented $4n$-dimensional Riemannian
manifold $(M,g)$ endowed with three global $g$-orthogonal parallel (hence integrable Kähler)
complex structures $I,J$ and $K$ compatible with the orientation such that $IJ=-JI=K$.
The corresponding pencil of complex structures $f~:~\T=\T(M,g,D)\to\PP^1$ is integrable and called the \textit{Calabi family} of $(M,g,D=(I,J,K))$.
Note that for $Sp(n)\subset SU(2n)$, each of these complex structure is Ricci-flat.

For example, starting with a compact \textit{holomorphic symplectic manifold} of complex dimension $2n$
(i.e. a compact complex Kähler manifold $X$ with a holomorphic symplectic $2$-form $\Omega$, 
hence of vanishing first Chern class) 
and a Kähler class $\kappa$, the theorem of Yau~\cite{yau} gives a unique Kähler metric $g$ in the Kähler class $\kappa$
with vanishing Ricci curvature. The form $\Omega$ is $g$-parallel by the Bochner principle,
showing that the holonomy of $g$ is contained in $U(2n)\cap Sp(2n,\C)$ that is the quaternionic unitary group $Sp(n)$,
and that $g$ is a hyperkähler metric.
The corresponding twistor space $f~:~\T(X,\Omega,\kappa)\to\PP^1$ is called \textit{the Calabi family} of $(X,\Omega,\kappa)$.
$$\xymatrix{
   \T \ar[d]^{\pi}\ar[r]^{f} &  \PP^{1}\ni u\\
m\in M\mbox{\ \  \ \ \ \ } & & }.$$ 
The Calabi family is differentiably isomorphic to the product $M\times\PP^{1}$, and the horizontal and vertical directions are given by $f$ and $\pi$.
Note in this case, that the horizontal distribution on $\T=\T(M,g,(I,J,K))$ is integrable.
In this special case, we choose the Riemannian metric $\G$ on $\T$ to be the product metric of $g$ with the spherical metric of radius $1$.

The manifold $X$ will be called \textit{irreducible holomorphic symplectic} if furthermore
$X$ is simply connected and $H^0(X,\Omega^2_X)$ is generated by the holomorphic symplectic $2$-form $\Omega$.

\section{Calabi families of hyperkähler manifolds}
 \label{sec:1}

We will work in this section on a hyperkähler manifold $(M,g,(I,J,K))$. 
We will  assume that the holonomy group is exactly $Sp(n)$ 
so that each $X_u:=f^{-1}(u)$ is an irreducible holomorphic symplectic manifold~\cite{beauville}.

\subsection{Computations of $d\wb$ and $d''\wb$}
We choose a complex coordinate $\zeta$ centred at the point parametrising the complex structure $I$.
The stereographic projection tell us that the complex structure $\J$ on
$T\T\simeq \pi^\star TM\oplus f^\star T\PP^1$ is given by
$$\J_{(m,\zeta)}=\left( \frac{1-|\zeta |^2}{1+|\zeta|^2}I_m
+\frac{i(\zeta -\ov\zeta)}{1+|\zeta|^2}J_m +\frac{\zeta
+\ov\zeta}{1+|\zeta|^2}K_m,i \right).$$ 
We choose local holomorphic coordinates $(z_1,\cdots z_n)$ on the complex manifold $(M,I)$.
Then, the forms $d\varphi_{2i-1}:=dz_{2i-1}-\zeta d\zb_{2i}$ and $d\varphi_{2i}:=dz_{2i}+\zeta d\zb_{2i-1}$ 
together with $d\zeta$
built a basis of the space of forms of type $(1,0)$ on $T_{(m,\zeta)}\T\otimes\C$.

This lead to the following description of the Hermitian form $\wb$ on $\T$
in terms of the closed form $w_I, w_J, w_K$ associated with the Kähler structures $I, J, K$
respectively. We choose $w_{\PP^1}$ to be $w_{\PP^1}=2\frac{id\zeta\wedge d\overline\zeta}{(1+\vert\zeta\vert^2)^2}$
of volume $4\pi$ as the sphere of radius $1$.

$$\wb_{(m,\zeta)}=\left( \frac{1-|\zeta |^2}{1+|\zeta|^2}w_I
+\frac{i(\zeta -\ov\zeta)}{1+|\zeta|^2}w_J +\frac{\zeta
+\ov\zeta}{1+|\zeta|^2}w_K,w_{\PP^1} \right).$$ 

We first compute its exterior derivative
 \begin{eqnarray*}
 d\wb&=&\frac{1}{(1+\vert\zeta\vert^2)^2}\Big(-2\overline\zeta
 w_I+i(1+\overline\zeta^2)w_J+(1-\overline\zeta^2)w_K\Big)\wedge d\zeta\\
 &&+
 \frac{1}{(1+\vert\zeta\vert^2)^2}\Big(-2\zeta
 w_I-i(1+\zeta^2)w_J+(1-\zeta^2)w_K\Big)\wedge d\overline\zeta.
 \end{eqnarray*}
 To extract its $(1,2)$ part, simply check 
from 
\begin{eqnarray*}
 w_I=\frac{i}{2}\sum_{j=1}^{n/2} dz_{2j-1}\wedge d\zb_{2j-1}+ dz_{2j}\wedge d\zb_{2j}\\
 w_J=\frac{1}{2}\sum_{j=1}^{n/2} dz_{2j-1}\wedge dz_{2j}    + d\zb_{2j-1}\wedge d\zb_{2j}\\
 w_K=\frac{-i}{2}\sum_{j=1}^{n/2} dz_{2j-1}\wedge dz_{2j}    - d\zb_{2j-1}\wedge d\zb_{2j}
\end{eqnarray*}
that
$$
d\wb=\frac{i}{(1+\vert\zeta\vert^2)^2}\left(\sum_{j}d\overline\varphi_{2j-1}\w d\overline\varphi_{2j}\w d\zeta
-\sum_{j}d\varphi_{2j-1}\w d\varphi_{2j}\w d\overline\zeta\right).
$$
Hence,
$$ \label{d``w} d''\wb=\frac{1}{(1+\vert\zeta\vert^2)^2}
\Big(-2\overline\zeta w_I+i(1+\overline\zeta^2)w_J+(1-\overline\zeta^2)w_K\Big)\wedge d\zeta.$$

\subsection{Kodaira-Spencer map and K\"ahler property}

The main drawback with the use of twistor spaces is that they are almost never of Kähler type, 
even under strong vanishing assumptions for the curvature of~$g$.
This defect can be quantified, at least for the natural metric on the twistor space $\T(M,g,(I,J,K))$, 
by the Kodaira-Spencer class.
 We now prove an intrinsic version of the previous formula~\ref{d``w}
\begin{theo}\label{domega-k}
Let $(\theta_1,\ldots,\theta_{4n})$ be a local orthonormal frame of $TM$.
 The exterior derivative $d\wb$ of the Hermitian form $\wb$ on the twistor space of a hyperkähler manifold
$(M,g,(I,J,K))$  vanishes on pure directional vectors
except when evaluated on two horizontal vectors and one vertical vector.
More precisely then, for a vertical vector $U\in\vc_{(m,u)}$,
\begin{eqnarray*}
d''\wb_{(m,u)}(U^h,\hh\theta_i^a,\hh\theta_j^a)&=&-2\Omega_u(\kappa_U(\theta_i^a),\kappa_U(\theta_j^a))
\end{eqnarray*}
where $\Omega_u$ is the holomorphic symplectic $(2,0)$-form on $X_u:=f^{-1}(u)$ 
and $\kappa_U$ is a closed $(0,1)$-form on $X_u$
with values in $TX_u$ that represents the Kodaira-Spencer class of the family $f$ at $u\in\PP^1$
in the direction $U$.
\end{theo}

\begin{proof}
  We first follow~\cite[proposition 25.7]{huybrechts}.
Consider  a path $\gamma(t)$ in the base $\PP^1$ starting at $u=I$ say, with derivative $U\in T\PP^1$. 
Over every point $m\in M$, there is a vertical lift that we may write as
$u_m(t)=I_m+t U_m+t^2\cdots$. Note that $U_m$ is the derivative in the direction $U$,
$\varphi_\star (U)\in so(T_mM)$,
 of the map $\varphi : \pi^{-1}(m)\to \SO (T_mM)$
that encodes the variation of complex structure on $T_mM$.
For small $t$, we write $T^{0,1}_mX_{u_m(t)}$ as the graph of a map $K(t)=t\kappa_U+t^2\cdots$
 from $T_m^{0,1}X_u$ to $T_m^{1,0}X_u$. Note that $\kappa_U$ seen as a $(0,1)$-form on $X_u$
 with values on $T^{1,0}X_u$ is closed 
by the integrability of $\T$ and has, as cohomology class, the Kodaira-Spencer class 
$\{\kappa_U\}\in H^1(TX_u)$.
For a vector $v\in T_m^{0,1}X_u$, we have the relation
$u_m(t)(v+K(t)(v))=-i(v+K(t)(v))$ whose first order term gives 
$I_m(\kappa_U(v))+U_m(v)=i\kappa_U(v)+U_m(v)=-i\kappa_U(v)$.
This shows that 
\begin{eqnarray*}\label{phi}\varphi_\star (U)(v)=U_m(v)=-2i\kappa_U(v).\end{eqnarray*}
Now, note that, because the horizontal and the vertical distributions are integrable and horizontal lifts commutes with vertical lifts,
the brackets occurring in the following computations vanish
\begin{eqnarray*}
d\wb(U,\hh\theta_i,\hh\theta_j)&=& U\cdot\wb(\hh\theta_i,\hh\theta_j)\\&&-
\wb\big([\hh\theta_i,\hh\theta_j],U\big)-\wb([U,\hh\theta_i],\hh\theta_j)+\wb([U,\hh\theta_j],\hh\theta_i)\\
&=&U\cdot g(u(\theta_i),\theta_j)
\end{eqnarray*}
so that
\begin{eqnarray*}
 d''\wb(U^h,\hh\theta_i^a,\hh\theta_j^a)
&=&g(\varphi_\star (U)(\theta_i^a),\theta_j^a)
=-2ig(\kappa_U (\theta_i^a),\theta_j^a)\\
&=&-2\omega_u(\kappa_U(\theta_i^a),\theta_j^a)=-2\Omega_u(\kappa_U(\theta_i^a),\kappa_U(\theta_j^a)).
\end{eqnarray*} 
The last equality is proved in~\cite{huybrechts}.
\end{proof}

\subsection{Computations of $\hess\wb$}
We recall a formula for the hessian of the Hermitian form~$\wb$ (\cite[section 8.4]{KV}).
Our computations follow from the expression of the $(1,2)$-part of the exterior derivative $d\wb$.
\begin{eqnarray*}
 idd''\wb&=&\frac{i}{(1+\vert\zeta\vert^2)^2}
\Big(-2d\overline\zeta w_I+2i\overline\zeta d\overline\zeta w_J-2\overline\zeta d\overline\zeta w_K\Big)\wedge d\zeta\\
&&-2\frac{i\zeta}{(1+\vert\zeta\vert^2)^3}\Big(-2\overline\zeta w_I+i(1+\overline\zeta^2)w_J+(1-\overline\zeta^2)w_K\Big)d\overline\zeta\wedge d\zeta
=\wb\w w_{\PP^1}
\end{eqnarray*}
 We get
\begin{theo}\label{theo2}
The hessian $\hess\wb$ of the Hermitian form of the Calabi family $\T(M,g,(I,J,K))$ of a hyperk\"ahler manifold $(M,g,(I,J,K))$
is given by
$$\hess \wb=\wb\w w_{\PP^1}.$$
\end{theo}

\subsection{Pseudo convexity of the cycle space in the case of $\R^4$}
We first study, as an example, the small deformations in a locally conformally flat situation.

The twistor space $\T(\R^4)$ of the flat Euclidean $\R^4$ is
described as a complex manifold as the total space of the rank two
vector bundle $\mathcal O(1)\oplus \mathcal O(1)$ on $\PP^1$. 
The differentiable product structure is given by the map
$$
\begin{array}{clc}
\mathcal O(1)\oplus \mathcal O(1)&\stackrel{\psi}{\lra}& \C^2\times\PP^1\\
(a\zeta+b,c\zeta+d)&\lms&(z_1,z_2,\zeta)
\end{array}
\textrm{ with }\left\{\begin{array}{l} 
\displaystyle z_1=\frac{\ov c \zeta+\ov d\vert\zeta\vert^2+a+b\zeta}{1+\vert\zeta\vert^2}\\
\displaystyle z_2=\frac{-\ov a\zeta-\ov b\vert \zeta\vert^2 +c+d\zeta}{1+\vert\zeta\vert^2}.
\end{array}\right.
$$
Twistor fibres are given by $z_1=constant$ and $z_2=constant$, that is $c=-\ov b$  and
$d=\ov a$.
The cycle space $C_1(\T(\R^4))$ is simply the vector space $H^0(\PP^1,\mathcal
O(1)\oplus\mathcal O(1))$ of holomorphic sections.
Irreducible cycles are parametrised in the form  $(a\zeta+b,c\zeta+d)$.
The volume function that can be computed as
$$\vol (\PP^1)\Big(1+\frac{\vert a-\ov d\vert^2+\vert \ov b+c\vert^2}{4}\Big)$$ 
achieves its minimum for twistor lines. To compute the Hessian of the Hermitian form, 
we consider a point $s\in C_1(\T(\R^4))$, 
where $C_s$ is parametrised by $(a\zeta+b,c\zeta+d)$. 
We consider a non-zero tangent vector
$n=(\alpha\zeta+\beta,\gamma\zeta+\delta)\in H^0(\PP^1,\mathcal O(1)\oplus\mathcal O(1))\simeq T_{s} C_1(\T(\R^4))$, 
whose norm is the flat Hermitian norm on $\C^2$ (for which $\mid\!\mid\frac{\partial}{\partial z_i}\mid\!\mid^2=\frac{1}{2}$) of $\psi_\star(n)$. 
Then theorem~\ref{theo2} gives :
$$
\hess\vol_{C_1(\T)}(C_s)(\vec{n},\ov{\vec{n}})=  \vol(\PP^1)
\frac{\vert\alpha\vert^2+\vert\beta\vert^2+\vert\gamma\vert^2+\vert\delta\vert^2}{4}>0
$$
which is coherent with the former expression.

From the description of the twistor space of the conformally flat $4$-sphere $\Sp^4=\R^4\cup\{\infty\}$
as $\PP^3$ and of the cycle space as the grassmannian of lines in $\PP^3$, we recover this strict pseudo-convexity
by the ampleness property of the Schubert divisor parametrising lines meeting the twistor line $\pi^{-1}(\infty)$.

\subsection{Convexity of the $1$-cycle space}
Let $(M,g, (I,J,K))$ be a hyperkähler manifold and $\T=\T((M,g,(I,J,K))\stackrel{f}{\to}\PP^1$ its Calabi family.
Let  $C^0_1(\T)$ be the component of the Barlet cycle space of $\T$ containing the twistor lines.

For every $s=[C_s]\in C^0_1(\T)$, we will identify a tangent vector $\vec{n}\in T_s C^0_1(\T)$ with a section $n$ of the normal sheaf
$N_{C_s/\T}$ of the $1$-cycle $C_s$ in the twistor space $\T$.
The intersection number of a cycle $C_s$ in $C^0_1(T)$ with a fibre of $f$ being constantly $1$, 
we infer that every member $C_s$ contains, outside an irreducible section $C'_s$ of the pencil $f$, 
a finite number $\sum H_j$ of horizontal rational curves.
The component $C'_s$ being a section of $f$, there is an horizontal lifting $\tilde{n}$ 
of the normal section $n$, whose norm is simply denoted by $\Vert n\Vert$. 
\begin{theo}\label{plurif}
The map $\vol : C^0_1(\T)\to\R$ is a continuous pluri-sub-harmonic exhaustion function. 
In particular, the cycle space $C^0_1(\T)$ is pseudo-convex.
More precisely,
\begin{eqnarray}
 \label{in}\hess_{C^0_1(\T)}\vol (C_s)(\vec{n},J\ov{\vec{n}})\geq\int_{C'_s} \Vert n'\Vert ^2 d vol\geq 0
\end{eqnarray}
where $C'_s$ is the irreducible component of the cycle $C_s$ that maps onto $\PP^1$ by the pencil map $f$.
\end{theo}

\begin{proof}
The volume function is gotten by integration of the Hermitian form $\wb$ on the smooth part of the cycles.
It is well-defined by a theorem of Lelong~\cite{lelong}.

By definition, a continuous function $\phi$ on an analytic space $Y$ is \textit{pluri-subharmonic}
if  every point of $Y$ has a neighbourhood $V$ that embeds in a complex ball $B$
where the function $\phi$  can be extended as a continuous pluri-sub-harmonic function.
By a theorem of Fornaess-Narasimhan~\cite{fo-na}, the map $\phi$ is pluri-sub-harmonic on $Y$
if and only if for every holomorphic maps $j:\Delta\to Y$ from the unit disc to $Y$,
the function $\phi\circ j$ is pluri-sub-harmonic.

Choose a cycle $C_0$, a tangent vector $\vec{n}\in T_{[C_0]} C^0_1(\T)$, and a family 
$$\xymatrix{
 \mathcal{C}   \ar[d]^{\Pi}\ar[r]^{\Gamma} & \T\\
0\in \Delta\ \ \ \ \ & & }$$
of cycles with this tangent vector $\vec{n}$ at the origin. 
Then, 
\begin{eqnarray*}
 \hess_{C^0_1(\T)}\vol(C_s)(\vec{n},J\ov{\vec{n}})
&=&\hess \Pi_\star\Gamma^\star \wb(\vec{n},J\ov{\vec{n}})
=\Pi_\star\Gamma^\star \hess\wb(\vec{n},J\ov{\vec{n}})\geq 0.
\end{eqnarray*}
by theorem~\ref{theo2}

For the irreducible image $C'_s$ of a section $\sigma~:\PP^1\to\T$ of the pencil $f$,
\begin{eqnarray*}
 \int_{C'_s}\Gamma^\star \hess\wb(\vec{n},J\ov{\vec{n}})
&=&\int_{\PP^1}\sigma^\star\hess\wb(\tilde{n},\J\ov{\tilde{n}})\\
&=&\int_{\C}\hess\wb(\sigma_\star\frac{\partial}{\partial \zeta},
\J\ov{\sigma_\star\frac{\partial}{\partial
\zeta}},\tilde{n},\J\ov{\tilde{n}})d\lambda_\C(\zeta)
\end{eqnarray*}
 where $\zeta$ is a complex coordinate on $\PP^1$
and $\tilde{n}$ is any lifting of $n$ under $T\T_{\mid C_s}\to N_{C_s/\T}$. 
For $C'_s$ is a section of the map $f$, 
the composed map $\hc_{\mid C'_s}\hookrightarrow
T\T_{\mid C'_s}\to  N_{C'_s/\T}$ is an isomorphism and we can assume that the lifting $\tilde{n}$ lies
in  $\hc_{\mid C'_s}$. Hence, by theorem~\ref{theo2} only
the vertical part $\frac{\partial}{\partial \zeta}$ of
$\sigma_\star\frac{\partial}{\partial \zeta}
=\sum_{i=1}^{4n}\frac{\partial}{\partial
\zeta}\sigma_i(\zeta)\frac{\partial}{\partial x_i}
\oplus\frac{\partial}{\partial \zeta}$ contributes : 
\begin{eqnarray*}
 \int_{C'_s}\Gamma^\star \hess\wb(\vec{n},J\ov{\vec{n}})
&=&\int_\C \hess\wb_{\sigma(\zeta)}(\frac{\partial}{\partial\zeta}, 
\J\ov{\frac{\partial}{\partial\zeta}},\tilde{n},\J\ov{\tilde{n}})d\lambda_\C(\zeta).
\end{eqnarray*}
 Theorem~\ref{theo2} gives
\begin{eqnarray*}
 \hess\wb_{\sigma(\zeta)}(\frac{\partial}{\partial\zeta},\J\ov{\frac{\partial}{\partial \zeta}},
\tilde{n},\J\ov{\tilde{n}}) d\lambda_\C(\zeta) 
&=&\| n\|^2_\hc
d\lambda_{\PP^1}(\zeta).
\end{eqnarray*}

As for the horizontal part $ H_j$, using a parametrisation by $\PP^1$, we get
\begin{eqnarray*}
\int_{H_j} \Gamma^\star\hess\vol(\vec{n},J\ov{\vec{n}})
&=&\int_{\PP^1} \hess\wb(h,\J\ov{h},\tilde{n},\J\ov{\tilde{n}})
\end{eqnarray*}
where $h$ is horizontal and 
where $\tilde{n}$ is any lifting of $n$ under $T\T_{\mid C_s}\to N_{C_s/\T}$. 
By the Kähler property of the fibres of $f$ or by theorem~\ref{theo2}, 
only the vertical part of the lifting is relevant,
and this contributes non-negatively to the hessian.

The map $\vol$ being a continuous  exhaustion~\cite{lieb},
we infer from its pluri-sub-harmonicity, that the cycle space $C^0_1(\T)$ is pseudo-convex.
\end{proof}

\begin{rema}
The inequality~(\ref{in}) displays the fact that, because a non zero tangent vector $\vec{n}\in T_0 C^0_1(\T)$ 
can have zero component $n'$ on the slanted component $C'_0$,
there can be compact families of horizontal
$1$-cycles, as for example, in the Hilbert scheme $Hilb^n(S)\supset Hilb^n(C)=\PP^n$ of a $K3$ surface 
that contains a smooth rational curve $C$. 
Nevertheless, in the case of $K3$ surfaces ($n=1$), Verbitsky~\cite{Verb3} proved that the component $C^0_1(\T)$
is in fact Stein.
\end{rema}


\section{Twistor spaces of $4$-dimensional anti-self dual Riemannian  manifolds}
\label{sec:2}
We consider in this whole section a $4$-dimensional anti-self dual Riemannian manifold $(M,g)$.
Let $\nabla^g$ be the Levi-Civita connection of $(M,g)$,
$\eta$ its connection $1$-form in a given frame with values in $so(TM)$
and $R$ its curvature tensor defined by
$R(X,Y)Z:=[\nabla^g_Y,\nabla^g_X]Z +\nabla^g_{[X,Y]}Z$. Recall that, with these conventions
 $R(X,Y)= -(d\eta+\eta\w\eta)(X,Y)$.
As an endomorphism of $\bigwedge^2TM=\bw^+\oplus\bw^-$ its decomposition is
$$
R=\left[\begin{array}{cc} W^+ +\frac{s}{12}Id&B\\
^tB&W^- +\frac{s}{12}Id
\end{array}\right]
$$
Here $B:\left\{\begin{array}{l}\bw^+\to\bw^-\\ \bw^-\to\bw^+\end{array}\right.$
is the trace-free Ricci tensor. It vanishes for Einstein's metrics.
The operator $W=W^++W^-$, called the Weyl operator, depends only on the conformal class of the Riemannian metric $g$ 
and  $s$ is the scalar curvature of $g$.
By a fundamental theorem of~\cite{AHS78} the almost complex structure $\J$
is integrable if and only if the metric $g$ on $M$ is anti-self-dual, that is $W^+=0$.
We will always assume this.
By the works of Trudinger, Aubin, and Schoen~\cite{schoen} on Yamabe problem, when $M$ is compact, we will always choose
a conformal representative of $g$ with constant scalar curvature.
This does not change the isomorphism class of $(\T,\J)$.

\subsection{Properties of type and directional decompositions}

\begin{lemm}\label{type}
 Given a positively oriented orthonormal frame $(\theta_1,\ldots,\theta_4)$ on an open set $\mathcal U$ of
$M$
 \begin{enumerate}
\item for all $(\alpha,\beta)$ in $\bw^+_\C\times\bw^-_\C$, the matrix bracket $[\alpha,\beta]$ (in fact
$[\phi^{-1}(\alpha),\phi^{-1}(\beta)])$ vanishes. 
\item
$\theta_j^h\w\theta_k^h\in\bw^+_\C$ and
$\theta_j^a\w\theta_k^a\in\bw^+_\C$ 
\item $\theta_i^h\w\theta_j^a\in\bw^-_\C\oplus Vect(\phi(u))_\C)$.
\end{enumerate}
\end{lemm}

\begin{proof}
 \begin{enumerate}
\item Any endomorphism $A\in\phi^{-1}(\bigwedge^+)\subset so(4)$ coming from a bivector of norm $\sqrt{2}$ can be described
as the left multiplication by a quaternion with a quaternion $q$ as the quaternion product $AX=q\cdot X$, 
and likewise any $B\in\phi^{-1}(\bigwedge^-)$ coming from a bivector of norm $\sqrt{2}$ can be described
as the right multiplication by a quaternion. The result now follows from the associativity
of the quaternion algebra.
 More explicitly note that the family  $ \left\{\begin{array}{l}
\,\theta_1\w\theta_2+\theta_3\w\theta_4\\
\,\theta_1\w\theta_3-\theta_2\w\theta_4\\
\,\theta_1\w\theta_4+\theta_2\w\theta_3
\end{array}
\right.$ is a basis of $\bigwedge^+$ and that
$\left\{\begin{array}{l}
\,\theta_1\w\theta_2-\theta_3\w\theta_4\\
\,\theta_1\w\theta_3+\theta_2\w\theta_4\\
\,\theta_1\w\theta_4-\theta_2\w\theta_3
\end{array}
\right. $ is a basis of $\bigwedge^-$.

\item At a point $p=(m,u)$, expanding we get
$$\begin{array}{lll}\theta_j^h\wedge\theta_k^h
&=&\displaystyle \frac{1}{4}(\theta_j-iu\theta_j)\wedge(\theta_k-iu\theta_k)\\
&=&\displaystyle \frac{1}{4}
\Big(\theta_j\wedge\theta_k-u\theta_j\wedge u\theta_k
-i(\theta_j\wedge u\theta_k+u\theta_j\wedge\theta_k)\Big)\\
 &=&\displaystyle \frac{1}{4}(Id-iu)(\theta_j\wedge\theta_k
-u\theta_j\wedge u\theta_k)\in\bigwedge\nolimits^+_\C.
 \end{array}$$
 The relation
 $\theta_i^a\wedge\theta_j^a\in\bw^+_\C$ follows by conjugation.

 \item Expanding again, we get
 $$\begin{array}{lll}\theta_j^h\wedge\theta_k^a&=&\displaystyle \frac{1}{4}
\Big(\theta_j\wedge\theta_k+u\theta_j\wedge u\theta_k
+i(\theta_j\wedge u\theta_k-u\theta_j\wedge\theta_k)\Big).
 \end{array}$$
We then check $\left\{\begin{array}{c}\theta_j\w\theta_k+
u\theta_j\w u\theta_k\in\bw^-_\C\oplus Vect(\phi(u))\\
\theta_j\w u\theta_k- u\theta_j\w \theta_k\in\bw^-_\C\oplus Vect(\phi(u))\end{array}\right.$
\end{enumerate}
\end{proof}

The data of a positively oriented orthonormal frame $(\theta_1, \ldots, \theta_4)$ 
on an open set $\mathcal U$ of $M$ defines a trivialisation 
$\T\supset\pi^{-1}(\mathcal U)\simeq\mathcal U\times\Sp^2$. 
The local coordinates of a point $p$ in $\T$ will be denoted by  $(m,u)$.
Because the fibre of $\pi$ over a point $m\in M$ is
$\{u\in SO(T_mM)/u^2=-Id\textrm{ and } u\gg 0\}$,
the vertical space $\vc_p$ at a point $p=(m,u)$ is given by
$$\vc_p=\{A\in so(T_mM)/Au+uA=0\}.$$
Let  $A:\mathcal{U}\to so(TM)$ be a section of the bundle of anti-symmetric endomorphisms.
We define $\wh{A} : \pi^{-1}(\mathcal{U})\to T\T$ to be the associated vertical vector field
computed with matrix brackets
$$\wh{A}(p)=\wh{A}(m,u)=[u,A(m)]\in \vc_p.$$
Note that these special vectors generate the vertical directions.
\begin{rema} The first easy property of lemma~\ref{type} will hugely simplify the forthcoming computations.
For example, if $A:\mathcal{U}\to so(TM)$ is a section and $A^+ : \mathcal{U}\to\bw^+$
its projection onto  $\bw^+$ then the associated vertical vector fields are equal
$\wh{A}=\wh{A^+}$. In particular, $\wh{R(\theta_i^h\w\theta_j^h)}=
\wh{(W^++\frac{s}{12}Id)(\theta_i^h\w\theta_j^h)}$.
Similarly, because $B$ maps $\bw^+$ to $\bw^-$, the vertical vector field $\wh{B(u)}$ vanishes.
Hence, for the vertical vector field $\wh{B(\theta_i^h\w\theta_j^a)}$, only the component in
$\bigwedge^-_\C$ of the vector $\theta_i^h\w\theta_j^a\in\bigwedge^-_\C\oplus Vect (u)_\C$ is relevant.
\end{rema}

Let $X : \mathcal{U}\to TM$ be a vector field on $\mathcal{U}$.
Its horizontal lifting $\hh(X)$ is a basic vector field
(i.e. $\pi_\star\hh{(X)}=X$ everywhere on $\pi^{-1}(\mathcal{U}$)).
In terms of the local trivialisation $\pi^{-1}(\mathcal U)\simeq\mathcal U\times\Sp^2$, 
the principal bundle $P_{so(4)}$ of positively oriented
orthonormal frames of $TM$  maps onto $\T$ by
\begin{eqnarray*}
P_{so(4)}&\to&\T\\ (m,(\theta_1, \ldots, \theta_4))&\mapsto&\left\{ \begin{array}{c}(m,u) \textrm{ such that }\\
Mat(u,(\theta_1, \ldots, \theta_4))=I=\left[\begin{array}{cccc}
0&-1&0&0\\
1&0&0&0\\
0&0&0&-1\\
0&0&1&0\end{array}\right]\end{array}\right..
\end{eqnarray*}
We infer that the horizontal lifted vector field $\hh{(X)}$ reads
$$
T\T\supset \hc_p\ni\hh{(X)}=X+\wh{\eta(X)}\in T\mathcal{U}\oplus T\Sp^2.
$$

\subsection{Bracket computations}
The following bracket computations of basic vector fields will be used again and again.
We first discuss according to vertical and horizontal directions.
\begin{lemm}\label{bracket}
 Given a positively oriented orthonormal frame $(\theta_1,\ldots,\theta_4)$ on an open set $\mathcal U$ of
$M$ and two sections $A$ and $B$ of $\mathcal{U}\to so(TM)$, 
the Lie brackets of the associated vector fields are computed by
$$\begin{array}{ccccc}
[\wh A,\wh B] &=&&& \wh{[A,B]}\cr
[ \hh{(\theta_i)}, \wh{A} ] &=&&&\wh{(\nabla^g_{\theta_i}A)}+\wh{ [\eta(\theta_i),A] } \cr
[\hh{(\theta_i)},\hh{(\theta_j)}] &=&\hh{[\theta_i,\theta_j]}&-& \wh{R(\theta_i\w\theta_j)}.
\end{array}
$$
\end{lemm}

\begin{proof}
 At a point $p=(m,u)$ of $\T$,
\begin{enumerate}
 \item $
[\wh A,\wh B]=\left[\left[u, A\right],\left[u,B\right]\right]
= \left[\left[u, A\right],B\right]-\left[\left[u,B\right],A\right]
=\left[u,\left[A,B\right]\right].
$
The map $A\mapsto \wh{A}$ is hence a morphism of Lie algebras.
\item First note that,
$[\theta_i,\wh{A}]=[\theta_i,[u,A]]=[u,\nabla^g_{\theta_i}A]=\wh{ \nabla^g_{\theta_i}A }$.
Hence,
$$
[\hh{(\theta_i)}, \wh A]=[\theta_i+\wh{\eta(\theta_i)},\wh A]
=\wh{\nabla^g_{\theta_i}A }+[\wh{\eta(\theta_i),A]}.$$
\item The result derives from previous remarks
$$\begin{array}{ccc}
[\hh{(\theta_i)},\hh{(\theta_j)}]&=&
[\theta_i+\wh{\eta(\theta_i)},\theta_j+\wh{\eta(\theta_j)}]\\
&=&[\theta_i,\theta_j]+\wh{\nabla^g_{\theta_i}\eta(\theta_j)}
-\wh{\nabla^g_{\theta_j}\eta(\theta_i)}+
[\wh{\eta(\theta_i)},\wh{\eta(\theta_j)}]\\
&=&[\theta_i,\theta_j]+\wh{d
\eta(\theta_i,\theta_j)}+\wh{\eta([\theta_i.\theta_j])}+
\wh{[\eta(\theta_i),\eta(\theta_j)]}\\
&=&[\theta_i,\theta_j]+\wh{\eta([\theta_i.\theta_j])}+\wh{\Big(d
\eta+\eta\w\eta\Big)} (\theta_i,\theta_j)
\\
&=&\hh{[\theta_i,\theta_j]}-\wh{R(\theta_i\w\theta_j)}.
\end{array}$$

\end{enumerate}
\end{proof}

\begin{rema}
This last relation shows that the curvature $R$ accounts for the lack of integrability
of the horizontal distribution $\hc$.
\end{rema}

We now discuss adding type considerations. In the next lemma, the exponent $t$ denotes either types $h$ or $a$.
\begin{lemm}\label{Jlinearite}\label{bracket-type}
 For every vertical vector field $U$ on $\T$, the following commutation relations hold
\begin{enumerate}
\item $\hc[\hc\theta_i,U]=0 \textrm{ and }\hc[\J\hc\theta_i,U]=-U(\hc\theta_i)$
\item $\mbox{} [\hc\theta_i,\J U]=\J[\hc\theta_i,U]\textrm{ and }\vc[\J\hc\theta_i,\J U]=\J\vc[\J\hc\theta_i,U]$
\item $\hc[\hc\theta_i^h,U^{t}]=\frac{i}{2}U^{t}(\hc\theta_i^{t})
\textrm{ and }\hc[\hc\theta_i^a,U^{t}]=-\frac{i}{2}U^{t}(\hc\theta_i^{t})$
\item $ [ \hc\theta_i, U^{t} ]= [ \hc\theta_i, U ]^{t}\textrm{ and }\vc [ \J\hc\theta_i,U^{t} ]
=(\vc [ \J\hc\theta_i,U ])^{t}$.
\end{enumerate}
\end{lemm}
\begin{proof}
We use the notations of lemma~\ref{bracket}.
In order to get the first equality, simply note that for any smooth function $f$,
using the previous lemma and the properties of the Lie brackets, the vector field $[\hc\theta_i,f\wh A]$ is vertical.
The second computation hence reduces to
\begin{eqnarray*}
  \hc[\J\hc\theta_i,U]&=&\hc[\hc(\sum u_{ji}\theta_j),U]\\
&=&-\sum (U\cdot u_{ji})\hc\theta_j=-\sum U_{ji}\hc\theta_j:=-U(\hc\theta_i).
\end{eqnarray*}
Use the bracket linearity to get
$\hc[\hc\theta_i^h,U]=\frac{i}{2} U(\hc\theta_i)=-\hc[\hc\theta_i^a,U]$.
Note that this is a tensor in $U$ so that
in particular, $$\hc[\hc\theta_i^h,U^a]=\frac{i}{2} U^a( \hc\theta_i)
=\frac{i}{2} \frac{Id+iJ}{2}U^a( \hc\theta_i)
=\frac{i}{2}U^a( \hc\theta_i^a).$$

 The third formula follows from the fact that the parallel transport along horizontal directions
respect the canonical metric and the orientation of the fibres, hence the vertical complex structures.
The forth follows from the fact that
 $\vc[\hc\theta_i,\J U]=\J\vc[\hc\theta_i,U]$ is a tensor in $\theta_i$.

 The last four follow by linearity.
\end{proof}

\begin{rema}
 The formula $ \hc[\J\hc\theta_i,U]=-U(\hc\theta_i)$ can be made more intrinsic by considering the map
$\varphi : \pi^{-1}(m)\to \SO (T_mM)$
that encodes the variation of complex structure on $T_mM$.
We find
\[\pi_\star  [\J\hc\theta_i,U]=-\varphi_\star (U)( \theta_i).\]
\end{rema}

\subsection{Computations of $d\wb$ and $d'\wb$}

Let  $\wb=\G(\J\cdot,\cdot)$ be the Hermitian form on the twistor space~$\T$.
Its exterior derivative is given by the following
\begin{prop}\label{prop:dwb}
The exterior derivative $d\wb$ of the Hermitian form $\wb$ on the
twistor space $\T$ of an anti-self dual Riemannian $4$-manifold
$(M,g)$
  vanishes on pure directional (i.e. horizontal or vertical) vectors
except when evaluated on two horizontal vectors and one vertical vector.
More precisely then,
\[\forall X,Y\in TM,\forall U\in\vc\ \
d\wb(U,\hh X,\hh Y)=\G\Big(\wh{(\frac{1}{2}Id-R)(X\w Y)},\J U\Big).\]
\end{prop}

\begin{proof}
The results of this proposition are well known and can be found for example in~\cite{Barto}.
 The usual formula for the exterior derivatives of a $2$-form reduces here by orthogonality
 using the bracket computations of lemma~\ref{bracket} to
\[
\begin{array}{ccc}
d\wb(U,\hh\theta_i,\hh\theta_j)&=& U\cdot\wb(\hh\theta_i,\hh\theta_j)-
\wb\big([\hh\theta_i,\hh\theta_j],U\big)\\
&=&U\cdot g(u(\theta_i),\theta_j)+
\G\big([\hh\theta_i,\hh\theta_j],\J U\big)\\
&=&-U_{ij}-\G\big(\wh{R(\theta_i\w\theta_j)},\J U\big).
\end{array}
\]
choosing vertical coordinates $(u_{ij})$ such that $u(\theta_j)=\sum u_{ij}\theta_i$.
Now set $E=\theta_i\w\theta_j$. From the definition of $\G$,
and the property $uU=-Uu$ of the vertical vector $U$one has
\begin{eqnarray}\label{GU}\nonumber
\G(\wh{\theta_i\wedge\theta_j},\J U)&=&-1/2tr\Big((uE-Eu)uU\Big)\\&=&-1/2tr(uEuU)-1/2tr(EU)\\
&=&-tr(EU)=-2U_{ij}.\nonumber
\end{eqnarray}
It remains to check the vanishing of all the other pure directional components.
As the fibres are of real dimension two, the $3$-form $d\wb$ restricts to zero on fibres.
Let $A, B :M\to so(TM)$ be two sections. In normal coordinates at the centre of which
the connection $1$-form $\eta$ vanishes
\[
\begin{array}{lll}
d\wb(\hh\theta_i,\wh A,\wh B)&=&\hh\theta_i\cdot\wb(\wh A,\wh B)-
\wb\Big([\hh\theta_i,\wh A],\wh B\Big)
+\wb\Big([\hh\theta_i,\wh B],\wh A\Big)\\
&=&\theta_i\cdot\wb(\wh A,\wh B)-\wb\Big(\wh{\nabla^g_{\theta_i} A},\wh B\Big)
+\wb\Big(\wh{\nabla^g_{\theta_i} B},\wh A\Big)\\
&=&\wb\Big(\wh{\nabla^g_{\theta_i}A},\wh B\Big)
+\wb\Big(\wh A,\wh{\nabla^g_{\theta_i}B}\Big)\\
&&-\wb\Big(\wh{\nabla^g_{\theta_i}A},\wh B\Big)
+\wb\Big(\wh{\nabla^g_{\theta_i}B},\wh A\Big)=0.
\end{array}
\]
Finally for a triple of horizontal lifts, still with normal coordinates,
\[
\begin{array}{lll}
d\wb(\hh\theta_i,\hh\theta_j,\hh \theta_k)&=& 
\hh\theta_i\cdot\wb(\hh\theta_j,\hh\theta_k)
-\hh\theta_j\cdot\wb(\hh\theta_i,\hh\theta_k)\\
&&+\hh\theta_k\cdot \wb(\hh\theta_i,\hh\theta_j)
-\wb([\hh\theta_i,\hh\theta_j],\hh\theta_k)\\
&&+\wb([\hh\theta_i,\hh\theta_k],\hh\theta_j)
-\wb([\hh\theta_j,\hh\theta_k],\hh\theta_i)\\
&=&
\theta_i\cdot g(u\theta_j,\theta_k)-\theta_j\cdot g(u\theta_i,\theta_k)\\&&
+\theta_k\cdot g(u\theta_i,\theta_j)-g(u[\theta_i,\theta_j],\theta_k)
\\&&+g(u[\theta_i,\theta_k],\theta_j)
-g(u[\theta_j,\theta_k],\theta_i) = 0
\end{array}
\]
for $\theta_i\cdot u=0$ and
for all the remaining quantities can be expressed in terms of $\nabla^g_{\theta_a}\theta_b=0$ only.
\end{proof}

This result leads to an expression for the $(2,1)$-part $d'\wb$ of
$d\wb$.
\begin{prop}~\label{d'w}
For all vertical vectors $U$, one has
\begin{enumerate}
\item $\displaystyle d'\wb(U^a,\hh\theta_i^h,\hh\theta_j^h)=
\displaystyle(\frac{s}{6}-1)U^a_{ij}$.
\item $d'\wb(U^h,\hh\theta_i^h,\hh\theta_j^a)
=-i\G\Big(\wh{B(\theta_i^h\w\theta_j^a)},U^h\Big)$.
\end{enumerate}
\end{prop}

\begin{proof}
\begin{enumerate}
\item Because $\theta_i^h\wedge\theta_j^h=\displaystyle
\frac{1}{4}(Id-iu)(\theta_i\wedge\theta_j-u\theta_i\wedge
 u\theta_j)$ belongs to $\bw^+_\C$ we infer by lemma~\ref{type}
$$\begin{array}{lll}
d'\wb(U^a,\wh\theta_i^h,\wh\theta_j^h)
&=&\frac{1}{4}\G\Big(\widehat{(\frac{1}{2}Id-R)(Id-iu)(\theta_i\w\theta_j-u\theta_i\w u\theta_j)},
\J U^a\Big)\\
&=&\frac{1}{4}\G\Big(\widehat{(\frac{1}{2}-\frac{s}{12})(Id-iu)(\theta_i\w\theta_j-u\theta_i\w u\theta_j)},
\J U^a\Big)\\ &&\quad{\textrm{ for }W^+=0}\\
&=&\frac{1}{2}(\frac{1}{2}-\frac{s}{12})\G\Big(\frac{Id-i\J}{2}\widehat{
(\theta_i\wedge\theta_j-u\theta_i\wedge u\theta_j)},\J U^a\Big)\\
&=&\frac{1}{2}(\frac{1}{2}-\frac{s}{12})\G\Big(\widehat{(\theta_i\w\theta_j-u\theta_i\w u\theta_j)},
\frac{Id+i\J}{2}\J U^a\Big)\\
&=&\frac{1}{4}(1-\frac{s}{6})\G\Big(\widehat{(\theta_i\w\theta_j-u\theta_i\w u\theta_j)},\J U^a\Big).
\end{array}$$
But we already found in~\ref{GU} that $\G(\widehat{\theta_i\wedge\theta_j},\J
U)= -2U_{ij}$. Writing $u\theta_i=\sum u_{ki}\theta_k$, we find
$$\G(\widehat{u\theta_i\wedge u\theta_j},\J
U)=-2u_{ki}u_{lj}U_{kl}=2(uUu)_{ij}=2U_{ij}$$ that leads to $
 d'\wb(U^a,\wh\theta_i^h,\wh\theta_j^h)=(\frac{s}{6}-1)U^a_{ij}.
$ \item Because $\theta_i^h\w\theta_j^a$ belongs to
$\bw^-_\C\oplus Vect(u)$, we infer by lemma~\ref{type} on the one
hand
$\wh{R(\theta_i^h\w\theta_j^a)}=\wh{B(\theta_i^h\w\theta_j^a)}$
and on the other hand
 $\wh{\theta_i^h\w\theta_j^a}=0$. This gives
\begin{eqnarray*}
d'\wb(U^h,\wh\theta_i^h,\wh\theta_j^a)&=&
\G\Big(\wh{(\frac{1}{2}Id-R)(\theta_i^h\w\theta_j^a)},JU^h\Big)
=-i\G\Big(\wh{B(\theta_i^h\w\theta_j^a)},U^h\Big).
\end{eqnarray*}
\end{enumerate}
\end{proof}

\begin{coro}
 The form $\wb$ is Kähler if and only if $R\vert_{\bw^+}=\displaystyle\frac{1}{2}Id_{\bw^+}$.
\end{coro}

\begin{proof} The vanishing $d\wb=0$ gives $B=0$ (cf $ii$) and  $s/12=1/2$ (cf. $i$).
The converse is straightforward.
\end{proof}

\begin{rema}
 This is the case for the round sphere $\Sp^4$ and the projective space $\C P^2$ with a Fubini-Study metric.
More generally, Hitchin~\cite{Hit81} actually proved that these are the only compact Kähler twistor spaces.
\end{rema}

\subsection{Computation of $\hess\wb$}

In this subsection, we will compute the real $4$-form
$id''d'\wb=idd'\wb$ of type $(2,2)$. We will express its values on
pure directional vectors. This theorem accounts for the main features we found.
\begin{theo} \label{theo1}
The hessian $\hess\wb$ of the Hermitian form $\wb$ on the twistor
space $\T(M,g)$ of an anti-self dual Riemannian $4$-manifold
$(M,g)$ with constant scalar curvature $s$
 is given on pure directions and pure types by the following formulae where
$\hh\theta_i$ are basic horizontal lifts  and $U_i$ vertical vectors,
 \begin{enumerate}
\item $$\hess\wb(U_1^h,U_2^h,U_3^a,U_4^a)=0$$
\item $$\hess\wb(U_1^h,U_2^h,U_3^a,\hh\theta_i^a)=\hess\wb(\hh\theta_i^h,U_3^h,U_1^a,U_2^a)=0$$
\item $$\hess\wb(\hh\theta_i^h,\hh\theta_j^h,U_1^a,U_2^a)=\hess\wb(U_1^h,U_2^h, \hh\theta_i^a,\hh\theta_j^a)=0$$
\item \begin{eqnarray}\hess\wb\Big(\hh\theta_i^h,U_1^h,\hh\theta_j^a,U_2^a\Big)&=&
-U_2^a\cdot\G\Big(\wh{B(\theta_i^h\w\theta_j^a)},U_1^h\Big)\nonumber \\
&&-\frac{i}{2}U_{2\;mj}^a\G\Big(\wh{B(\theta_m^a\w\theta_i^h)},U_1^h\Big)\nonumber\\
&& +\frac{1}{2}(\frac{s}{6}-1)(U_2^a.U_1^h)_{ij}\label{1}\end{eqnarray}
\item $$\hess\wb(\hh\theta_i^h,\hh\theta_j^h,U^a,\hh\theta_k^a)=
\hess\wb(U^h,\hh\theta_k^h,\hh\theta_i^a,\hh\theta_j^a)=0$$
\item \begin{eqnarray}\hess\wb(\hh\theta_i^h,\hh\theta_j^h,\hh\theta_k^a,\hh\theta_l^a)
&=&\G\Big(\wh{B(\theta_j^h\w\theta_l^a)},
\wh{B(\theta_i^h\w\theta_k^a)}\Big)\nonumber
\\&&-\G\Big(\wh{B(\theta_i^h\w\theta_l^a)},\wh{B(\theta_j^h\w\theta_k^a)}\Big)\nonumber\\
&&-i(\frac{s}{6}-1)\frac{s}{12}(\wh{\theta_k^a\wedge\theta_l^a})^a_{ij}.\label{2}
\end{eqnarray}
\end{enumerate}
\end{theo}

\begin{proof}
\begin{enumerate}
\item   reflects the facts that the vertical distribution is integrable
and that the metric on the fibres is Kähler.
\item The non vanishing of $d'\wb$ requires two horizontal vectors.
The integrability of the vertical distribution hence shows the results.

\item
By the usual formula for the exterior derivative, omitting zero terms, we get
\begin{eqnarray*}
\lefteqn{d''d'\wb(\hc\theta_i^h,\hc\theta_j^h,U_1^a,U_2^a)}\\
&=&U_1^a\cdot d'\wb(\hc\theta_i^h,\hc\theta_j^h,U_2^a)
-U_2^a\cdot d'\wb(\hc\theta_i^h,\hc\theta_j^h,U_1^a)\\&&-d'\wb([U_1^a,U_2^a],\hc\theta_i^h,\hc\theta_j^h)\\
&&+d'\wb([\hc\theta_i^h,U_1^a],\hc\theta_j^h,U_2^a)-d'\wb([\hc\theta_i^h,U_2^a],\hc\theta_j^h,U_1^a)\\
&&+d'\wb([\hc\theta_j^h,U_1^a],\hc\theta_i^h,U_2^a)-d'\wb([\hc\theta_j^h,U_2^a],\hc\theta_i^h,U_1^a).
\end{eqnarray*}
From proposition~\ref{d'w} and lemma~\ref{bracket-type}, we infer that
the terms \\ $d'\wb([\hc\theta_i^h,U_1^a],\hc\theta_j^h,U_2^a)
=d'\wb(\hc[\hc\theta_i^h,U_1^a]^h,\hc\theta_j^h,U_2^a)$
in the last two lines vanishes for type reason.
As the scalar curvature is constant, the proposition~\ref{d'w} leads to
\begin{eqnarray*}
 \lefteqn{U_1^a.d'\wb(\hh\theta_i^h,\hh\theta_j^h,U_2^a)
-U_2^a.d'\wb(\hh\theta_i^h,\hh\theta_j^h,U_1^a)
-d'\wb([U_1^a,U_2^a],\hh\theta_i^h,\hh\theta_j^h)}\\
&&=(\frac{s}{6}-1)\Big(U_1^a.(U_2^a)_{ij}-U_2^a.(U_1^a)_{ij}-[U_1^a,U_2^a]_{ij}\Big)=0
\end{eqnarray*}
The second equality follows by conjugation.

\item By the usual formula for the exterior derivative, omitting zero terms, we get
\begin{eqnarray*}
\lefteqn{d''d'\wb\Big(\hh\theta_i^h,U_1^h,\hh\theta_j^a,U_2^a\Big)}\\
&=& U_2^a\cdot d'\wb\Big(U_1^h,\hh\theta_i^h,\hh\theta_j^a\Big)
-d'\wb\Big([\hh\theta_i^h,U_2^a],U_1^h,\hh\theta_j^a\Big)\\
&&-d'\wb\Big([\hh\theta_j^a,U_2^a],\hh\theta_i^h,U_1^h\Big)
+d'\wb\Big([\hc\theta_j^a,U_1^h],\hh\theta_i^h,U_2^a\Big)\\
&=&-i U_2^a\cdot\G\Big(\wh{B(\theta_i^h\w\theta_j^a)},U_1^h\Big)
-d'\wb\Big(\hh[\hh\theta_i^h,U_2^a]^h,U_1^h,\hh\theta_j^a\Big)\\
&&-d'\wb\Big(\hc[\hh\theta_j^a,U_2^a]^a,\hh\theta_i^h,U_1^h\Big)
+d'\wb\Big(\hc[\hh\theta_j^a,U_1^h]^h,\hh\theta_i^h,U_2^a\Big).
\end{eqnarray*}
From lemma~\ref{bracket-type}, we infer that the second term vanishes for type reasons,
and that for the third term $\hc[\hh\theta_j^a,U_2^a]=
\displaystyle-\frac{i}{2}U^a_{2\;mj}\hh\theta_m^{a}$.
Hence
$$
-d'\wb\Big([\hh\theta_j^a,U_2^a],\hh\theta_i^h,U_1^h\Big)
=\frac{1}{2}U^a_{2\;mj}\G\Big(\wh{B(\theta_m^a\w\theta_i^h)},U_1^h\Big).
$$
From lemma~\ref{bracket-type}, we infer that for the forth term
$\hc[\hh\theta_j^a,U_1^h]=-\displaystyle\frac{i}{2}
U^h_{1\;mj}\hh\theta_m^{h}$.
 Hence,
\begin{eqnarray*}
\lefteqn{d'\wb\Big([\hh\theta_j^a,U_1^h],\hh\theta_i^h,U_2^a\Big)}\\
&=&-\frac{i}{2}d'\wb\Big(U^h_{1\;mj}\hh\theta_m^h,\hh\theta_i^h,U_2^a\Big)\\
&=&-\frac{i}{2}(\frac{s}{6}-1) U^h_{1\;mj}U_{2\;mi}^a
=\frac{i}{2}(\frac{s}{6}-1)(U_2^aU_1^h)_{ij}.
\end{eqnarray*}

\item    Still from the formula of the exterior derivative
\begin{eqnarray*}
\lefteqn{d''d'\wb\Big(\hh\theta_i^h,\hh\theta_j^h,U^a,\hh\theta_k^a\Big)}\\
&=& -\hh\theta_k^a.d'\wb\Big(\hh\theta_i^h,\hh\theta_j^h,U^a\Big)
+d'\wb\Big(\vc[\hh\theta_k^a,U^a],\hh\theta_i^h,\hh\theta_j^h\Big)
\\
&&+d'\wb\Big(\vc[\hh\theta_i^h,U^a],\hh\theta_j^h,\hh\theta_k^a\Big)
 -d'\wb\Big(\vc[\hh\theta_j^h,U^a],\hh\theta_i^h,\hh\theta_k^a\Big)\\
&& +d'\wb\Big(\hc[\hh\theta_j^h,\hh\theta_k^a],\hh\theta_i^h,U^a\Big)
-d'\wb\Big(\hc[\hh\theta_i^h,\hh\theta_k^a],\hh\theta_j^h,U^a\Big).
\end{eqnarray*}
From proposition~\ref{d'w} we can write
 $$\hh\theta_k^a.d'\wb\Big(\hh\theta_i^h,\hh\theta_j^h,U^a\Big)=
\hh\theta_k^a.\Big((\frac{s}{6}-1)U^a_{ij}\Big)=\hh\theta_k^a.\Big((\frac{s}{6}-1)U^a_{ij}\Big)=0$$
computed in normal coordinates centred at a point $m$.
In such coordinates, we can choose $U=\wh A$ with furthermore
$\nabla^g_{\theta_i}A=0$ at $m$.
By lemma~\ref{Jlinearite}, $\vc[\hh\theta_k^h,U^a]=(\vc[\hh\theta_k^h,U])^a
=(\wh{(\nabla^g_{\theta_k}A)}+\wh{ [\eta(\theta_k),A] })^a =0$.
Now, the vanishing of the third and forth terms, follows from
 lemma~\ref{Jlinearite}, after which $\vc[\hh\theta_i^h,U^a]$ is of type~$(0,1)$.
Still at the centre $m$ of normal coordinates , we have $\hc
[\wh\theta_i^h,\wh\theta_k^a]=0$ because
$[\theta_i,\theta_j]=\nabla^g_{\theta_i}\theta_j-\nabla^g_{\theta_j}\theta_i=0$.
As $\hh\theta_i\cdot u=0$, we conclude
$\hc[\wh\theta_j^h,\wh\theta_k^a]=0$.

\item Again from the formula of the exterior derivative and from~\ref{d'w}
\begin{eqnarray*}
\lefteqn{d''d'\wb\Big(\hh\theta_i^h,\hh\theta_j^h,\hh\theta_k^a,\hh\theta_l^a\Big)}\\
&=&
d'\wb\Big([\hh\theta_i^h,\hh\theta_k^a],\hh\theta_j^h,\hh\theta_l^a\Big)
+d'\wb\Big([\hh\theta_j^h,\hh\theta_l^a],\hh\theta_i^h,\hh\theta_k^a\Big)
\\
 &&-d'\wb\Big([\hh\theta_j^h,\hh\theta_k^a],\hh\theta_i^h,\hh\theta_l^a\Big)
-d'\wb\Big([\hh\theta_i^h,\hh\theta_l^a],\hh\theta_j^h,\hh\theta_k^a\Big)
\\
&&-d'\wb\Big([\hh\theta_k^a,\hh\theta_l^a],\hh\theta_i^h,\hh\theta_j^h\Big).
\end{eqnarray*}
As $\theta_i^h\wedge\theta_k^a$ belongs to $\bw^-_\C\oplus
Vect(u)$ and $\vc[\hh\theta_i^h,\hh\theta_k^a]$ is a tensor, we
find
 $
[\hh\theta_i^h,\hh\theta_k^a]_V=-\wh{R(\theta_i^h\wedge\theta_k^a)}
=-\wh{B(\theta_i^h\wedge\theta_k^a)}$ using lemma~\ref{type}.
 This leads to
\begin{eqnarray*}
\lefteqn{d'\wb\Big([\hh\theta_i^h,\hh\theta_k^a],\hh\theta_j^h,\hh\theta_l^a\Big)+
d'\wb\Big([\hh\theta_j^h,\hh\theta_l^a],\hh\theta_i^h,\hh\theta_k^a\Big)}\\
&=&-d'\wb\Big(\wh{B(\theta_i^h\wedge\theta_k^a)^h},\hh\theta_j^h,\hh\theta_l^a\Big)
-d'\wb\Big(\wh{B(\theta_j^h\wedge\theta_l^a)^h},\hh\theta_i^h,\hh\theta_k^a\Big)\\
&=&i\G\Big(\wh{B(\theta_j^h\wedge\theta_l^a)},\wh{B(\theta_i^h\wedge\theta_k^a)^h}\Big)
+i\G\Big(\wh{B(\theta_i^h\wedge\theta_k^a)},\wh{B(\theta_j^h\wedge\theta_l^a)^h}\Big)\\
&=&i\G\Big(\wh{B(\theta_j^h\wedge\theta_l^a)^a},\wh{B(\theta_i^h\wedge\theta_k^a)^h}\Big)
+i\G\Big(\wh{B(\theta_i^h\wedge\theta_k^a)^a},\wh{B(\theta_j^h\wedge\theta_l^a)^h}\Big)\\
&=&i\G\Big(\wh{B(\theta_j^h\w\theta_l^a)},
\wh{B(\theta_i^h\w\theta_k^a)}\Big).
\end{eqnarray*}
where we used the orthogonality of two $(1,0)$ vectors. For the
last term, as $\theta_k^a\w\theta_l^a$ belongs to $\bw^+_\C$ we
find
 $\vc[\hh\theta_k^a,\hh\theta_l^a]=-\wh{R(\theta_k^a\wedge\theta_l^a)}
=\displaystyle-\frac{s}{12}\wh{\theta_k^a\wedge\theta_l^a}$.
Finally,
\begin{eqnarray*}
\lefteqn{-d'\wb\Big([\hh\theta_k^a,\hh\theta_l^a],\hh\theta_i^h,\hh\theta_j^h\Big)}\\
&=&-d'\wb\Big(\vc[\hh\theta_k^a,\hh\theta_l^a]^a,\hh\theta_i^h,\hh\theta_j^h\Big)
=
(\frac{s}{6}-1)\frac{s}{12}(\wh{\theta_k^a\wedge\theta_l^a})^a_{ij}.
\end{eqnarray*}

\end{enumerate}
\end{proof}

A detailed analysis of parts~(\ref{1}) and ~(\ref{2}) shows that the twistor construction here
does not provide examples of manifolds $(X,\omega)$ strong KT without $\omega$ being Kähler.
We recover a version of the result of Verbitsky~(\cite[corollary 3.4]{Verb3}).
\begin{coro}
 The form $\wb$ is  $\hess$-closed  if and only if it is $d$-closed
if and only if $s=6$.
\end{coro}

\subsection{Convexity of the $1$-cycle space}
In the case of Einstein manifolds, the trace-free Ricci tensor $B$ vanishes so that only parts~(\ref{1}) and ~(\ref{2}) occur. 
We first study the sign of part~(\ref{1}).
At the point $(m,I)$, a vertical vector $U$ writes $U=aJ+bK$ so that
$U^aU^h=\frac{1}{4}(U+iIU)(U-iIU)=-\frac{1}{2}(a^2+b^2)(Id+iI)$.
The sign of  $\hess\wb\Big(\hh\theta_j^h,U^h,\hh\theta_j^a,U^a\Big) $
is hence the sign opposite to that of  $\displaystyle \frac{s}{6}-1$.

For part~(\ref{2}), a simple computation, using that the anti-symmetric endomorphism associated with $\theta_k^a\wedge\theta_l$
sends $\theta_k$ on $\theta_l$ and $\theta_l$ on $-\theta_k$, leads to $(\wh{\theta_k^a\wedge\theta_l^a})^a_{kl}=\frac{i}{2}(1-u_{kl}^2)>0$,
so that the sign of part~(\ref{2}) is that of $(\frac{s}{6}-1)\frac{s}{12}$.

We find, using~\cite[proposition~1]{barlet} as a general argument to get the pluri-sub-harmonicity from the semi-positivity of the hessian
of $\wb$,
\begin{coro}\label{cor:einstein-convex}
 The hessian $\hess\wb$ of the Hermitian form $\wb$ on the twistor
space $\T=\T(M,g)$ of an anti-self dual Einstein $4$-manifold $(M,g)$ with non-positive constant scalar curvature $s$
is semi-positive. If furthermore $M$ is compact, the volume function on the $1$-cycle space is a continuous pluri-sub-harmonic exhaustion function.
\end{coro}

\section{Twistor spaces of quaternionic K\"ahler manifolds}
\label{sec:3}
In this section,  $(M,g,D)$ will be a quaternionic Kähler $4n$-manifold with constant scalar curvature $s$.
\subsection{Computations of $d\wb$ and $d'\wb$}
We begin with the exterior derivative of the natural Hermitian form.
\begin{prop}
The exterior derivative $d\wb$ of the Hermitian form $\wb$ on the
twistor space $\T$ of the quaternionic Kähler manifold $(M,g,D)$
  vanishes on pure directional (i.e. horizontal or vertical) vectors
except when evaluated on two horizontal vectors and one vertical vector.
More precisely for all $ X,Y$ in $TM$ and $U$ in $\vc$
\begin{eqnarray*}
d\wb(U,\hh X,\hh Y)&=&\G\Big(\wh{(\frac{1}{2}Id-R)(X\w Y)},\J U\Big)\\
&=&\Big(1-\frac{s}{n(n+2)}\Big)\G(UX,Y)
\end{eqnarray*}
where $R$ denotes the curvature of the restriction of the Levi-Civita connection to the rank three sub-bundle~$D$.
\end{prop}
\begin{proof}
Over a open set $\mathcal U$ of $M$ trivialising $\T=\T(M,g,D)\to M$, the lifting of a vector field $X$ on $M$ reads
$$
T\T\ni \hc (X)=X+\wh{\eta(X)}\in \hc\oplus\vc.
$$
The first equality is hence a formal analog of proposition~\ref{prop:dwb}.

We choose a point $(m,I)$ in the twistor space $\T$. Every vertical vector $U$ is of the form $U=aJ+bK$ and hence
$\J U=IU=-bJ+aK$, where $(I,J,K)$ is a direct orthogonal basis of $D$ of vectors of norm $2$. 
Recall now
\begin{lemm} (\cite[lemma 14.40]{Bes87})\label{fonda}
For all vectors  $(X,Y)\in TM$, with $c=\frac{s}{2n(n+2)}$ the following holds
$$
\begin{array}{rcr}
\,[I,R(X\w Y)]&=&-\gamma(X,Y) J+\beta(X,Y) K\\
\,[J,R(X\w Y)]&=&\gamma(X,Y) I-\alpha(X,Y) K\\
\,[K,R(X\w Y)]&=&-\beta(X,Y) I+\alpha(X,Y) J
\end{array}
\textrm{ where } \left\{
\begin{array}{l}
\alpha(X,Y)=c g(IX,Y)\\
\beta(X,Y)=c g(JX,Y)\\
\gamma(X,Y)=c g(KX,Y).\\
\end{array}
\right.
$$
\end{lemm}
Note that
$\alpha(X,Y)=\frac{2}{n+2}r(IX,Y)=\frac{2}{n+2}\frac{s}{4n} g(IX,Y)=c g(IX,Y)$.
This lemma encodes the Einstein property of the metric $g$ and simplifies the previous expression.
In fact, $\wh{R(X,Y)}=\big[I,R(X,Y)\big]=-cg(KX,Y)J+cg(JX,Y)K$. Then,
\begin{eqnarray*}
 \G\Big(\wh{R(X,Y)},\J U\Big)&=&
2 bcg(KX,Y)+2acg(JX,Y))\\
 &=&2c\G\Big((aJ+bK)X,Y\Big)=\frac{s}{n(n+2)}\G(UX,Y).
\end{eqnarray*}
A formula analog to formula~\ref{GU} gives the value $\G(\wh{X\w Y},\J U)=-2g(X,UY)=2g(UX,Y)$,
we conclude $\G\Big(\wh{(\frac{1}{2}Id-R)(X\w Y)},\J
U\Big)=\Big(1-\frac{s}{2n(n+2)}\Big)\G(UX,Y).$
\end{proof}

It then follows by linearity and anti-commutation $UJ=-JU$, that the $(2,1)$-part of the exterior derivative of $\wb$ reads
\begin{prop} For a vertical vector $U\in\vc$, and an orthonormal frame $(\theta_i)$ of $TM$, 
\begin{enumerate}
\item $\displaystyle
d'\wb(U^a,\hh\theta_i^h,\hh\theta_j^h)=-\Big(1-\frac{s}{n(n+2)}\Big)U^a_{ij}$
\item $d'\wb(U^h,\hh\theta_i^h,\hh\theta_j^a)= 0.$
\end{enumerate}
\end{prop}

\subsection{Computation of $\hess\wb$}
The previous paragraph showed that for quaternion K\"ahler manifolds,
the computations follows the lines of the $4$ dimension Einstein case (with vanishing trace-free Ricci operator $B$).
We find
\begin{theo} \label{theo3}
The hessian $\hess\wb$ of the Hermitian form $\wb$ on the twistor
space $\T=\T(M,g,D)$ of a quaternionic Kähler $4n$-manifold $(M,g,D)$ with constant scalar curvature $s$
 is given on pure directions and pure types by the following formulae where
$\hh\theta_i$ are basic horizontal lifts  and $U_i$ vertical vectors,
 \begin{enumerate}
\item $$\hess\wb(U_1^h,U_2^h,U_3^a,U_4^a)=0$$
\item $$\hess\wb(U_1^h,U_2^h,U_3^a,\hh\theta_i^a)=\hess\wb(\hh\theta_i^h,U_3^h,U_1^a,U_2^a)=0$$
\item $$\hess\wb(\hh\theta_i^h,\hh\theta_j^h,U_1^a,U_2^a)=\hess\wb(U_1^h,U_2^h, \hh\theta_i^a,\hh\theta_j^a)=0$$
\item \begin{eqnarray}\label{6}\hess\wb\Big(\hh\theta_i^h,U_1^h,\hh\theta_j^a,U_2^a\Big)=\frac{1}{2}(\frac{s}{n(n+2)}-1)(U_2^a.U_1^h)_{ij}\end{eqnarray}
\item $$\hess\wb(\hh\theta_i^h,\hh\theta_j^h,U^a,\hh\theta_k^a)=
\hess\wb(U^h,\hh\theta_k^h,\hh\theta_i^a,\hh\theta_j^a)=0$$
\item \begin{eqnarray}\label{3}\hess\wb(\hh\theta_i^h,\hh\theta_j^h,\hh\theta_k^a,\hh\theta_l^a)
&=&-i(\frac{s}{n(n+2)}-1)(\wh{R(\theta_k^a\wedge\theta_l^a)})^a_{ij}\\
&=&-i(\frac{s}{n(n+2)}-1)\frac{s}{2n(n+2)}(\wh{\theta_k^a\wedge\theta_l^a})^a_{ij}.\nonumber\end{eqnarray}
\end{enumerate}
\end{theo}

As in the previous setting,
\begin{coro}
 The form $\wb$ is $\hess$-closed  if and only if it is $d$-closed
if and only if $s=n(n+2)$.
\end{coro}
\begin{proof}
 At a point $(m,I)$ for $U=J$, one has $U^a=\displaystyle\frac{J+iK}{2}$
and $U^h=\displaystyle\frac{J-iK}{2}$ so that
$U^aU^h=-\displaystyle\frac{Id+iI}{2}\not=0$.
It follows from the identity ~(\ref{6}) that  $\hess\wb\Big(\hh\theta_i^h,U_1^h,\hh\theta_j^a,U_2^a\Big)$ vanishes
for all $ij$ if and only if $s=n(n+2)$. 
\end{proof}

\subsection{Convexity of the $1$-cycle space}
We study the signs of non-zero terms.
We first study the sign of part~(\ref{3}).
At the point $(m,I)$, a vertical vector $U$ writes $U=aJ+bK$ so that
$U^aU^h=\frac{1}{4}(U+iIU)(U-iIU)=-\frac{1}{2}(a^2+b^2)(Id+iI)$.
Hence, $(U^aU^h)_{jj}=-\frac{1}{2}(a^2+b^2)=-\frac{1}{4}\Vert U\Vert^2=-\frac{1}{2}\Vert U^h\Vert^2$.
The sign of  
\begin{eqnarray*}
\hess\wb\Big(\hh\theta_j^h,U^h,\hh\theta_j^a,U^a\Big)
& =& \frac{1}{2}(\frac{s}{n(n+2)}-1)(U^a.U^h)_{jj}\\&=&-\frac{1}{4}(\frac{s}{n(n+2)}-1)\Vert U^h\Vert^2\\
 &=&-\frac{1}{2}(\frac{s}{n(n+2)}-1)\Vert\theta^h_j\Vert^2\Vert U^h\Vert^2
\end{eqnarray*}
is hence the sign opposite to that of  $\displaystyle \frac{s}{n(n+2)}-1$.

To study the sign of part~(\ref{6}), we have to again make use of the fundamental lemma~\ref{fonda}
\begin{eqnarray*}
\wh{R(\theta_j^a\w\theta_k^a)}&=&\big[I,R(\theta_i^a\w\theta_j^a)\big]=-cg(K\theta_j^a,\theta_k^a)J+cg(J\theta_j^a,\theta_k^a)K\\
&=&\frac{1}{4}\left(-cg\Big(K(\theta_j+iI\theta_j),\theta_k+iI\theta_k\Big)J+cg\Big(J(\theta_j+iI\theta_j),\theta_k+iI\theta_k\Big)K\right)\\
&=&\frac{c}{2}(K_{jk}+iJ_{jk})J-2(J_{jk}-iK_{jk})K\end{eqnarray*}
So, $\displaystyle\wh{R(\theta_j^a\w\theta_k^a)}_{jk}=\frac{ic}{2}(J_{jk}^2+K_{jk}^2)$.
It follows that
\begin{eqnarray*}
\hess\wb(\hh\theta_j^h,\hh\theta_k^h,\hh\theta_j^a,\hh\theta_k^a)&=&
-i(\frac{s}{n(n+2)}-1)\wh{R(\theta_j^a,\theta_k^a)}_{jk}\\
&=&\frac{c}{2}(\frac{s}{n(n+2)}-1)(J_{jk}^2+K_{jk}^2)\\
&=&\frac{1}{2}(\frac{s}{n(n+2)}-1)\frac{s}{2n(n+2)}(J_{jk}^2+K_{jk}^2).
\end{eqnarray*}
So, the sign of part~(\ref{6}) is
that the sign  of $\displaystyle (\frac{s}{n(n+2)}-1)\frac{s}{2n(n+2)}$.

In the case of non-positive scalar curvature, we get the
\begin{coro}\label{pseudo-convex}
 The hessian $\hess\wb$ of the Hermitian form $\wb$ on the twistor
space $\T=\T(M,g,D)$ of a quaternionic Kähler $4n$-manifold $(M,g,D)$ with non-positive constant scalar curvature $s$
is semi-positive. If furthermore $M$ is compact, the volume function on the $1$-cycle space is a continuous pluri-sub-harmonic exhaustion function.
\end{coro}

In the case of vanishing scalar curvature we recover the formula of theorem~\ref{theo2}
\begin{eqnarray*}
\hess\wb\Big(\hh\theta_j^h,U^h,\hh\theta_j^a,U^a\Big)
 &=&\frac{1}{2}\Vert\theta^h_j\Vert^2\Vert U^ h\Vert^2
\end{eqnarray*}
 up to the factor $\frac{1}{2}$ that accounts for the change of the radius of the vertical spheres from $1$ to $\sqrt{2}$.

In the case of positive scalar curvature, our considerations are compatible
with the compacity of the cycle space. Moreover, we find that there is exactly one way of adjusting the volume (i.e.
the scalar curvature) of the base manifold in order to make the volume function constant.

\backmatter

\providecommand{\bysame}{\leavevmode ---\ }
\providecommand{\og}{``}
\providecommand{\fg}{''}
\providecommand{\smfandname}{\&}
\providecommand{\smfedsname}{\'eds.}
\providecommand{\smfedname}{\'ed.}
\providecommand{\smfmastersthesisname}{M\'emoire}
\providecommand{\smfphdthesisname}{Th\`ese}

\end{document}